\documentclass[a4paper,11pt,reqno]{amsart}

\usepackage{amsfonts,latexsym,rawfonts,amsmath,amssymb,amsthm, mathrsfs, lscape}
\usepackage[all]{xy}
\usepackage[english]{babel}
\usepackage[utf8]{inputenc}
\usepackage{xfrac}

\usepackage{array, tabularx}

\usepackage{bm}
\bmdefine{\bX}{X}
\bmdefine{\ba}{$\alpha$}

\usepackage{setspace}
\usepackage{textcomp}

\usepackage{xparse}
\usepackage{soul}

\newcommand{\referenza}{}

\newtheorem{thm}{Theorem}[section]
\newtheorem*{thm*}{Theorem \referenza}
\newtheorem{cor}[thm]{Corollary}
\newtheorem*{cor*}{Corollary \referenza}

\newtheorem*{lem*}{Lemma \referenza}

\newtheorem*{clm*}{Claim \referenza}
\newtheorem{prop}[thm]{Proposition}
\newtheorem*{prop*}{Proposition \referenza}

\newtheorem*{conj*}{Conjecture \referenza}

\newtheorem{rmk}[thm]{Remark}

\theoremstyle{plain}

\numberwithin{equation}{section}

\def \O {\mathcal O}

\def \Q {\mathbb Q}
\def \R {\mathbb R}
\def \C {\mathbb C}
\def \Z {\mathbb Z}

\newcommand{\sspace}{\_}

\renewcommand{\Re}{\mathsf{Re}\,}
\renewcommand{\Im}{\mathsf{Im}\,}

\allowdisplaybreaks

\title[Rigidity of OT manifolds]{Rigidity of Oeljeklaus-Toma manifolds}

\author{Daniele Angella}
\address[Daniele Angella]{
Dipartimento di Matematica e Informatica ``Ulisse Dini''\\
Universit\`a degli Studi di Firenze\\
viale Morgagni 67/a\\
50134 Firenze, Italy
}
\email{daniele.angella@gmail.com}
\email{daniele.angella@unifi.it}

\author{Maurizio Parton}
\address[Maurizio Parton]{
Dipartimento di Economia\\
Universit\`a di Chieti-Pescara\\
viale della Pineta 4\\
65129 Pescara, Italy
} 
\email{parton@unich.it}

\author{Victor Vuletescu}
\address[Victor Vuletescu]{
Faculty of Mathematics and Informatics\\
University of Bucharest\\
Academiei st. 14\\
Bucharest, Romania
}
\email{vuli@fmi.unibuc.ro}

\thanks{During the preparation of this work, the first-named author has been supported by the Project PRIN ``Varietà reali e complesse: geometria, topologia e analisi armonica'', by the Project FIRB ``Geometria Differenziale e Teoria Geometrica delle Funzioni'', by SNS GR14 grant ``Geometry of non-Kähler manifolds'', by project SIR 2014 AnHyC ``Analytic aspects in complex and hypercomplex geometry" (code RBSI14DYEB), by ICUB Fellowship for Visiting Professor, and by GNSAGA of INdAM.
\\
The second-named author is supported by the Project PRIN ``Varietà reali e complesse: geometria, topologia e analisi armonica'' and by GNSAGA of INdAM.\\
The third-named author is partially supported by CNCS UEFISCDI, project number PN-II-ID-PCE-2011-3-0118}

\date{\today}

\begin{document}

\begin{abstract}
We prove that Oeljeklaus-Toma manifolds of simple type are rigid, and that any line bundle on an Oeljeklaus-Toma manifold is flat.
\end{abstract}

\maketitle

\section*{Introduction}
Oeljeklaus-Toma manifolds are complex non-K\"ahler manifolds. They have been introduced in \cite{oeljeklaus-toma} as counterexamples to a conjecture by I. Vaisman concerning locally conformally K\"ahler metrics.
Because of their construction using number fields techniques, many of their properties are encoded in the algebraic structure \cite{oeljeklaus-toma,vuletescu,dubickas}, and their class is well-behaved under such properties \cite{sverbitsky-curves, sverbitsky-surfaces}. They generalize Inoue-Bombieri surfaces in class VII \cite{inoue, tricerri}, and they are in fact solvmanifolds \cite{kasuya-blms}.

For example, K. Oeljeklaus and M. Toma proved in \cite[Proposition 2.5]{oeljeklaus-toma}, among other results, that the line bundles $K_X^{\otimes k}$ for $k\in \Z$ are flat. In this note, we use tools both from the number theoretic construction and from complex analysis and analytic geometry to prove more generally that:

\renewcommand{\referenza}{\ref{thm:flat}}
\begin{thm*}
Any line bundle on an Oeljeklaus-Toma manifold is flat.
\end{thm*}

Recently, A. Otiman and M. Toma \cite{otiman-toma} performed a precise and complete study of the Dolbeault cohomology of certain domains contained in Cousin groups, that includes our analysis below. We give here a self-contained proof in our specific context, that actually fits in the same perspective as \cite{vogt-crelle, otiman-toma}.

With techniques very similar to those developed in \cite{otiman-toma}, we get the following vanishing result:
\renewcommand{\referenza}{\ref{thm:vanishing}}
\begin{thm*}
Let $X(K,U)$ be an Oeljeklaus-Toma manifold, and $\rho\colon U \to \C^*$ be a faithful representation. Consider $L_\rho$ its associated flat holomorphic line bundle. Then one has $H^1(X;L_\rho)=0$ unless $\rho=\bar{\sigma}_ i^{-1}$ for some $i \in \{t+1,\dots t+s\}$.
\end{thm*}

As a corollary, we get rigidity, in the sense of the theory of deformations of complex structures of Kodaira-Spencer-Nirenberg-Kuranishi, for Oeljeklaus-Toma manifolds of simple type. Note that for the Inoue-Bombieri surface $S_M$, this is proven by Inoue in \cite[Proposition 2]{inoue}.
Here, by saying that the Oeljeklaus-Toma manifold $X(K,U)$ associated to the algebraic number field $K$ and to the admissible group $U$ is {\em of simple type}, we understand that there exists no proper intermediate field extension $\Q\subset K^\prime \subset K$ with $U\subseteq\mathcal O_{K^\prime}^{*,+}$, that is, there exists no holomorphic foliation of $X(K,U)$ with a leaf isomorphic to $X(K^\prime,U)$ \cite[Remark 1.7]{oeljeklaus-toma}.

\renewcommand{\referenza}{\ref{cor:rigid}}
\begin{cor*}
Oeljeklaus-Toma manifolds of simple type are rigid.
\end{cor*}

{\small
\noindent{\sl Acknowledgments.}
The authors would like to thank Oliver Braunling, Liviu Ornea, Alexandra Otiman, Fabio Podestà, Alberto Saracco, Tatsuo Suwa, Matei Toma, Adriano Tomassini, Valentino Tosatti for interesting comments, suggestions, and discussions.
Many thanks also to the first Reviewer for pointing out an error in a previous version of the paper, and to the second Reviewer for many valuable comments. 
Part of this work has been completed during the stay of the first-named author at Universitatea din Bucure\c{s}ti with the support of an ICUB Fellowship: many thanks for the invitation and for the warm hospitality.
}

\section{Oeljeklaus-Toma manifolds}
Oeljeklaus-Toma manifolds \cite{oeljeklaus-toma} provide a beautiful family of examples of compact complex non-K\"ahler manifolds, generalizing Inoue-Bombieri surfaces \cite{inoue}. In this section, we briefly recall the construction and main properties of Oeljeklaus-Toma manifolds from \cite{oeljeklaus-toma}. See \cite{ornea-vuletescu} and \cite[Section 6 of arXiv version]{parton-vuletescu} for more details and algebraic number theory background.

Let $K$ be an algebraic number field, namely, a finite extension of $\Q$. Then $K\simeq\Q[X]\slash(f)$ as $\Q$-algebras, where $f\in\Q[X]$ is a monic irreducible polynomial of degree $n=[K:\Q]$.
By mapping $X\mod(f)$ to a root of $f$, the field $K$ admits $n=s+2t$ embeddings in $\C$, more precisely, $s$ real embeddings $\sigma_1,\ldots,\sigma_s\colon K\to \R$, and $2t$ complex embeddings $\sigma_{s+1},\ldots,\sigma_{s+t},\sigma_{s+t+1}=\overline\sigma_{s+1},\ldots,\sigma_{s+2t}=\overline\sigma_{s+t} \colon K\to\C$.
Note that, for any choice of natural numbers $s$ and $t$, there is an algebraic number field with $s$ real embeddings and $2t$ complex embeddings, \cite[Remark 1.1]{oeljeklaus-toma}.

Denote by $\mathcal O_K$ the ring of algebraic integers of $K$, namely, elements of $K$ satisfying monic polynomial equations with integer coefficients. Note that, as a $\Z$-module, $\mathcal O_K$ is free of rank $n$.
Denote by $\mathcal O_K^*$ the multiplicative group of units of $\mathcal O_K$, namely, invertible elements in $\mathcal O_K$. By the Dirichlet's unit theorem, $\mathcal O_K^*$ is a finitely generated Abelian group of rank $s+t-1$.
Denote by $\mathcal O_K^{*,+}$ the subgroup of finite index of $\mathcal O_K^*$ whose elements are totally positive units, namely, units being positive in any real embedding: $u\in\mathcal{O}_K^*$ such that $\sigma_j(u)>0$ for any $j\in\{1,\ldots,s\}$.

Let $\mathbb H:=\left\{z\in\C : \Im z>0\right\}$ denote the upper half-plane. On $\mathbb H^s\times\C^t$, consider the following actions:
\begin{equation}\label{eq:T}
\begin{split}
T\colon &\mathcal{O}_K \circlearrowleft \mathbb H^s\times\C^t, \\
T_a(w_1,\ldots,w_s,z_{s+1},\ldots,z_{s+t}) &:= (w_1+\sigma_1(a),\ldots,z_{s+t}+\sigma_{s+t}(a)),
\end{split}
\end{equation}
and
\begin{equation}\label{eq:rot}
\begin{split}
R\colon &\mathcal{O}_K^{*,+} \circlearrowleft \mathbb H^s\times\C^t, \\
R_u(w_1,\ldots,w_s,z_{s+1},\ldots,z_{s+t}) &:= (w_1\cdot \sigma_1(u),\ldots,z_{s+t}\cdot \sigma_{s+t}(u)).
\end{split}
\end{equation}

For any subgroup $U\subset\mathcal{O}_K^{*,+}$, one has the fixed-point-free action $\mathcal{O}_K\rtimes U\circlearrowleft\mathbb H^s\times\C^t$.
One can always choose an admissible subgroup \cite[page 162]{oeljeklaus-toma}, namely, a subgroup such that the above action is also properly discontinuous and cocompact. In particular, the rank of admissible subgroups is $s$. Conversely, when either $s=1$ or $t=1$, every subgroup $U$ of $\mathcal{O}_K^{*,+}$ of rank $s$ is admissible.

One defines the {\em Oeljeklaus-Toma manifold} associated to the algebraic number field $K$ and to the admissible subgroup $U$ of $\mathcal O_K^{*,+}$ as
$$ X(K,U) := \left. \mathbb H^s\times\C^t \middle\slash \mathcal O_K\rtimes U \right.$$

In particular, for an algebraic number field $K$ with $s=1$ real embeddings and $2t=2$ complex embeddings, choosing $U=\mathcal O_K^{*,+}$ we obtain that $X(K,U)$ is an Inoue-Bombieri surface of type $S_M$ \cite{inoue}.

The Oeljeklaus-Toma manifold $X(K,U)$ is called {\em of simple type} when there exists no proper intermediate field extension $\Q\subset K^\prime \subset K$ with $U\subseteq\mathcal O_{K^\prime}^{*,+}$, that is, there exists no holomorphic foliation of $X(K,U)$ with a leaf isomorphic to $X(K^\prime,U)$ \cite[Remark 1.7]{oeljeklaus-toma}.

Oeljeklaus-Toma manifolds are non-K\"ahler solvmanifolds \cite[\S6]{kasuya-blms}, with Kodaira dimension $\kappa(X)=-\infty$ \cite[Proposition 2.5]{oeljeklaus-toma}. Their first Betti number is $b_1=s$, and their second Betti number in the case of simple type is $b_2={s\choose 2}$ \cite[Proposition 2.3]{oeljeklaus-toma}. Their group of holomorphic automorphisms is discrete \cite[Corollary 2.7]{oeljeklaus-toma}.
The vector bundles $\Omega^1_X$, $\Theta_X$, $K_X^{\otimes k}$ for $k\in \Z$ are flat and admit no non-trivial global holomorphic sections \cite[Proposition 2.5]{oeljeklaus-toma}.
Other invariants are computed in \cite[Proposition 2.5]{oeljeklaus-toma} and \cite{tomassini-torelli}. Recently, their Dolbeault cohomology is described in \cite{otiman-toma}.
Oeljeklaus-Toma manifolds do not contain either any compact complex curve \cite[Theorem 3.9]{sverbitsky-curves}, or any compact complex surface except Inoue-Bombieri surfaces \cite[Theorem 3.5]{sverbitsky-surfaces}.
When $t=1$, they admit a locally conformally K\"ahler structure \cite[page 169]{oeljeklaus-toma}, with locally conformally K\"ahler rank either $\frac{b_1}{2}$ or $b_1$ \cite[Theorem 5.4]{parton-vuletescu}. This is the Tricerri metric \cite{tricerri} in case $s=1$ and $t=1$.

In the case $t\geq2$, no locally conformally K\"ahler metrics are known to exist, so far. The fact that such an Oeljeklaus-Toma manifold does not carry a locally conformally K\"ahler metric was proven for $s=1$ already in the original paper \cite[Proposition 2.9]{oeljeklaus-toma}, later extended to the case $s<t$ by \cite[Theorem 3.1]{vuletescu}, and eventually widely extended to almost all cases by \cite[Theorem 2]{dubickas}.
Most likely, in the case $t\geq2$, no Oeljeklaus-Toma manifold carries a locally conformally K\"ahler metric.
However, note that Oeljeklaus-Toma manifolds admit no Vaisman metrics \cite[Corollary 6.2]{kasuya-blms}.

\section{Flatness of line bundles on Oeljeklaus-Toma manifolds}\label{sec:flat}

Let $X=X(K,U)$ be the Oeljeklaus-Toma manifold associated to the algebraic number field $K$ and to the admissible subgroup $U\subseteq \mathcal O_K^{*,+}$. Let $s$ denote the number of real embeddings of $K$ and $2t$ the number of complex embeddings of $K$.

For a better understanding of the cohomology of $X$, we start from its very definition, in the form of the following diagram of fibre-bundles:
\begin{equation}\label{dg:main}
\xymatrix{
\tilde X := \mathbb H^s \times \C^t \ar[rd]^{\mathcal O_K} \ar[dd]_{\pi_1(X)=\mathcal O_K \rtimes U} & \\
& X^{\text{ab}} := \left. \mathbb H^s \times \C^t \middle\slash \mathcal O_K \right. \ar[ld]^{U} \\
X := \left. \mathbb H^s \times \C^t \middle\slash \mathcal O_K\rtimes U \right.
}
\end{equation}

\begin{thm}\label{thm:flat}
Any line bundle on an Oeljeklaus-Toma manifold is flat.
\end{thm}

\begin{proof}
Equivalence classes of line bundles on $X$ are given by $H^1(X;\mathcal O^*_X)$, and the flat ones are given by the image of the map $n\colon H^1(X;\C^*_X)\to H^1(X;\mathcal O^*_X)$ induced by $\C_X\hookrightarrow\mathcal{O}_X$.
The statement is then equivalent to prove that the map
$$
n\colon H^1(X;\C^*_X) \to H^1(X;\mathcal O^*_X)
$$
is an isomorphism.

The map $n$ appears naturally from the following morphism of short exact sequences of sheaves:
 $$ \xymatrix{
  0 \ar[r] & {\Z}_X \ar[r] & \mathcal O_X \ar[r] & \mathcal O^*_X \ar[r] & 0 \\ 
  0 \ar[r] & {\Z}_X \ar[r] \ar@{=}[u] & {\C}_X \ar[r] \ar@{^{(}->}[u] & {\C}^*_X \ar[r] \ar@{^{(}->}[u] & 0 \\ 
 } $$
and the corresponding induced morphism of long exact sequences in cohomology:
\begin{equation}\label{eq:five}
\resizebox{.9\hsize}{!}{
$
\xymatrix{
H^1(X;\Z_X) \ar[r] & H^1(X;\mathcal O_X) \ar[r] & H^1(X;\mathcal O^*_X) \ar[r] & H^2(X;\Z_X) \ar[r] & H^2(X;\mathcal O_X) \\
H^1(X;\Z_X) \ar[r] \ar@{=}[u] & H^1(X;\C_X) \ar[r] \ar[u]_{m} & H^1(X;\C^*_X) \ar[r] \ar[u]_{n} & H^2(X;\Z_X) \ar[r] \ar@{=}[u] & H^2(X;\C_X) \ar[u]_{q}.
}
$
}
\end{equation}
By the Five Lemma, it suffices to prove that, in diagram~\eqref{eq:five}:
\begin{itemize}
\item[\textbf{(H1)}] $m$ is an isomorphism;
\item[\textbf{(H2)}] $q$ is injective.
\end{itemize}

\begin{rmk}
Notice that both of the claims are now proven in \cite[Corollary 4.6, Corollary 4.9]{otiman-toma} as a consequence of a more general description of Dolbeault cohomology of certain domains contained in Cousin groups.
For our aim, we will need a description of $H^1(X^{\text{ab}};\mathcal O_{X^{\text{ab}}})$ as in \cite[Theorem 3.1]{otiman-toma}: for the sake of completeness, we give here below a self-contained argument in our simpler case, in the same line of thought.
Compare also previous partial results by A. Tomassini and S. Torelli \cite{tomassini-torelli} for the case $s=2$ real places and $2t=2$ complex places.
\end{rmk}

\begin{proof}[Proof of Claim (H1).]
To prove that $m$ is an isomorphism, consider the following exact sequence of sheaves:
 $$ \xymatrix{
  0 \ar[r] & \C_X \ar[r] & \mathcal O_X \ar[r] & d\mathcal O_X \ar[r] & 0
 } $$
 and the induced exact sequence in cohomology:
 $$ \xymatrix{
  H^0(X; d\mathcal O_X) \ar[r] & H^1(X; \C_X) \ar[r]^m & H^1(X; \mathcal O_X) .
 } $$
Note that $H^0(X; d\mathcal O_X)=0$, since $H^0(X;\Omega^1_X)=0$ by \cite[Proposition 2.5]{oeljeklaus-toma}. Therefore $m$ is injective.
 Using the fact that $\dim_\C H^1(X;\C_X)=s$ \cite[Proposition 2.3]{oeljeklaus-toma}, it suffices to prove that $\dim_\C H^1(X;\mathcal O_X) = s$.

In order to describe the cohomology of $X$, we use diagram \eqref{dg:main}: we would like to relate the cohomology of $X$ with the $U$-invariant cohomology of $X^{\text{ab}}$. In what follows, we use group cohomology and the Lyndon-Hochschild-Serre spectral sequence to accomplish this task.

In general, whenever one has a map $\pi\colon\tilde X\to X=\tilde X/G$, for a free and properly discontinuous action of a group $G$ on $\tilde X$, and a sheaf $\mathcal{F}$ on $X$, there is an induced map
\begin{equation}\label{eq:groupcoho}
H^p(G; H^0(\tilde X;\pi^*\mathcal{F}))\to H^p(X;\mathcal{F}),
\end{equation}
where the first is the group cohomology of $G$ with coefficients in the $G$-module $H^0(\tilde X;\pi^*\mathcal{F})$, see for instance \cite[Appendix at page 22]{mumford}.
If, moreover, $\pi^*\mathcal{F}$ is acyclic over $\tilde X$, then the map \eqref{eq:groupcoho} is an isomorphism.

Using the previous argument on the $\mathcal O_K\rtimes U$ and the $\mathcal O_K$ maps in diagram \eqref{dg:main}, with $\mathcal{F}=\mathcal{O}_X$ and $\mathcal{F}=\mathcal{O}_{X^{\text{ab}}}$ respectively,
and noting that $\mathcal O_{\tilde X}$ is acyclic over $\tilde X$,
we obtain the isomorphisms
\begin{equation*}
\begin{split}
&H^p(\mathcal O_K\rtimes U;H^0(\tilde X;\mathcal{O}_{\tilde X})) \simeq H^p(X;\mathcal{O}_X)\\
\quad\text{and}\quad &H^p(\mathcal O_K;H^0(\tilde X;\mathcal{O}_{\tilde X})) \simeq H^p(X^{\text{ab}};\mathcal{O}_{X^{\text{ab}}}).
\end{split}
\end{equation*}

Hereafter, for the sake of clearness of notation, we denote the $\mathcal O_K\rtimes U$-module $R:=H^0(\tilde X;\mathcal{O}_{\tilde X})$. The previous isomorphisms are then written as
\begin{equation}\label{eq:Hpiso}
H^p(\mathcal O_K\rtimes U;R)\simeq H^p(X;\mathcal{O}_X)\quad\text{ and }\quad H^p(\mathcal O_K;R)\simeq H^p(X^{\text{ab}};\mathcal{O}_{X^{\text{ab}}}).
\end{equation}

The extension $\mathcal O_K\hookrightarrow\mathcal O_K\rtimes U\twoheadrightarrow U$ gives the associated Lyndon-Hochschild-Serre spectral sequence
\[
E^{p,q}_2=H^p(U;H^q(\mathcal{O}_K;R))\Rightarrow H^{p+q}(\mathcal O_K\rtimes U;R) ,
\]
and the cohomology five-term exact sequence yields
\begin{equation*}
\xymatrix{
&&& 0 \ar[ld] \\
& & H^1(U;H^0(\mathcal O_K;R)) \ar[r] & H^{1}(\mathcal O_K\rtimes U;R) \ar[lld] \\
& H^1(\mathcal O_K;R)^U \ar[r] & H^2(U;H^0(\mathcal O_K;R)) \ar[r] & H^{2}(\mathcal O_K\rtimes U;R).
}
\end{equation*}
From~\eqref{eq:Hpiso}, we get $H^0(\mathcal O_K;R)\simeq H^0(X^{\text{ab}};\mathcal O_{X^{\text{ab}}})=\C$, see \cite[Lemma 2.4]{oeljeklaus-toma}, whence $H^{1}(U;H^0(\mathcal O_K;R)) = \C^{\mathrm{rk}(U)} = \C^s$. Applying again \eqref{eq:Hpiso}, the cohomology five-term exact sequence becomes
\begin{equation}\label{eq:five-term-spectral-O}
\xymatrix{
0 \ar[r] & \C^s \ar[r] & H^{1}(X;\mathcal O_X) \ar[r] & H^1(X^{\text{ab}};\mathcal O_{X^{\text{ab}}})^U \ar[dll] \\
& H^2(U;\C_U) \ar[r] & H^{2}(X;\mathcal O_X).
}
\end{equation}
Therefore, the statement will follow by proving that
\begin{description}
\item[{(H1')}] $H^1(X^{\text{ab}};\mathcal O_{X^{\text{ab}}})^U=0$.
\end{description}
This is a consequence of the more general result in \cite[Theorem 3.1]{otiman-toma}, and we give here below an argument.

More precisely, we first claim that any class in $[\alpha] \in H^1(X^{\text{ab}};\mathcal O_{X^{\text{ab}}})$ has a unique flat representative, $[\alpha] \ni \sum_j c_j d\bar z_j$, where $c_j\in\mathbb C$ are constant, and $(z_1,\ldots,z_t)$ denote the coordinates in $\mathbb C^t$. Since $U\ni u$ acts on $[\alpha]$ by \eqref{eq:rot}, namely, $R_u^*[\alpha]=[\sum_j c_j \cdot \bar\sigma_j(u) d\bar z_j]\neq[\alpha]$ unless $[\alpha]=0$, then it follows that $H^1(X^{\text{ab}};\mathcal O_{X^{\text{ab}}})^U=0$.

We prove now the claim. As suggested in \cite[Proof of Lemma 3.2 at page 5]{otiman-toma}, we look at $X^{\text{ab}}$ as a holomorphic fibre bundle over a complex torus $B=\mathbb C^t/T$ with fibres $F$ being logarithmically convex Reinhardt domains in $(\mathbb C^*)^{s}$ , whence Stein. Recall that the matrix $T$ of periods of $B$ is obtained by putting the matrix $P$ of $\Lambda=\{(\sigma_1(a),\dots, \sigma_{s+t}(a)) \;\vert\; a\in \O_K\}$ into a more convenient form, see \cite[page 4]{AK}.
We consider the Borel-Serre spectral sequence for the Dolbeault cohomology of the holomorphic fibre bundle $F \hookrightarrow X^{\text{ab}} \twoheadrightarrow B$ with Stein fibres \cite{lupacciolu}:
$$ {}^{p,q} E_2^{s,p+q-s} \simeq \bigoplus_{\ell} H^{\ell,s-\ell}_{\overline\partial}(B; H^{p-\ell,q-s+\ell}_{\overline\partial}(F)) \Rightarrow H^{p,q}_{\overline\partial}(X^{\text{ab}}) . $$
Since
$$ 0 = {}^{0,0}E^{-1,1}_2 \stackrel{d_2}{\to} {}^{0,1}E_2^{1,0} \stackrel{d_2}{\to} {}^{0,2}E_2^{3,-1}=0 , $$
we compute
$$ {}^{0,1}E_\infty^{1,0} = {}^{0,1}E_2^{1,0} \simeq H^{0,1}_{\overline\partial}(B; \mathcal{O}_F) ; $$
moreover,
$$ {}^{0,1}E_2^{0,1} \simeq H^{0,0}_{\overline\partial}(B; H^{0,1}_{\overline\partial}(F)) = 0 . $$
Therefore, we can compute
$$ H^{0,1}_{\overline\partial}(X^{\text{ab}}) \simeq \bigoplus_{s+t=1} {}^{0,1}E_\infty^{s,t} = {}^{0,1}E_2^{1,0} \simeq H^{0,1}_{\overline\partial}(B; \mathcal{O}_F) . $$

From now on, for simplicity of calculations, we will denote the elements of $\mathbb H^s$ by $w=(w_1,\dots, w_s)$ and respectively those in $\mathbb C^t$ by $z=(z_1,\dots, z_t).$
We then notice that a class in $H^{0,1}_{\overline\partial}(X^{\text{ab}})$ is represented by
$$ \alpha = \sum_j f_j(w,z) d\bar z_j $$
where $f_j(w,z)$ are smooth functions in $(w,z)\in \mathbb H^s \times \mathbb C^t$, periodic with respect to $T$ (as defined in \eqref{eq:T}), and such that $\overline\partial\alpha=0$; that is, $f_j$ are holomorphic in $w$ and satisfy $\frac{\partial f_j}{\partial \bar z_k}=\frac{\partial f_k}{\partial \bar z_j}$ for any $j\neq k$.
We claim that we can find $c_j(\Im w) \in \mathcal{C}^\infty(\Im \mathbb H^s)$ and $g(w,z) \in \mathcal C^\infty(\mathbb H^s \times \mathbb C^t)$, holomorphic in $w$ and periodic with respect to $T$, such that
$$ \frac{\partial g}{\partial \bar z_j}(w,z)+c_j(\Im w)=f_j(w,z). $$
Since $c_j(\Im w)$ are holomorphic and periodic with respect to $T$, they are constant, and we can take $c_j=f_j(\Im w,0)$. We are then reduced to find $g$ such that
$$ \frac{\partial g}{\partial \bar z_j}=f_j-c_j . $$

We name coordinates $(v, a):=(\Im w, (\Re w, \Re z, \Im z)) \in \mathbb R^s \times \mathbb R^{s+2t}$.
By using Fourier expansion, we can write
$$ f_j(v,a)-c_j = \sum_{L \in \mathbb Z^{s+2t}\setminus 0} f_{j,L}(v) \exp \left( 2\pi\sqrt{-1} \langle AL \vert a \rangle \right) $$
where $A$ is the matrix whose columns are the coefficients of the lattice $\mathcal O_K$ with respect to the standard basis of $\mathbb R^{s+2t}$, so, $A$ has algebraic coefficients.
Here $L\neq0$ because of $f_j(\Im w,0)-c_j=0$.

The condition $\frac{\partial f_j}{\partial\bar z_k}-\frac{\partial f_k}{\partial\bar z_j}=0$ rewrites as: for any $L$, for any $j\neq k$,
\begin{equation}\label{eq:compatibility}
f_{j,L} \cdot \left( (AL)_{s+k} + \sqrt{-1} (AL)_{s+t+k} \right)
=
f_{k,L} \cdot \left( (AL)_{s+j} + \sqrt{-1} (AL)_{s+t+j} \right) .
\end{equation}

Analogously, we expand
$$ g(v,a) = \sum_{L \in \mathbb Z^{s+2t}} g_{L}(v) \exp \left( 2\pi\sqrt{-1} \langle AL \vert a \rangle \right) . $$
Denoting $z_j$ by $x_j+\sqrt{-1}y_j$, we compute
\begin{eqnarray*}
\frac{\partial g}{\partial \bar z_j}
&=& \frac{1}{2}\left(\frac{\partial g}{\partial x_j}+\sqrt{-1}\frac{\partial g}{\partial y_j}\right) = \frac{1}{2}\left(\frac{\partial g}{\partial a_{s+j}}+\sqrt{-1}\frac{\partial g}{\partial a_{s+t+j}}\right) \\
&=& \pi\sqrt{-1}\cdot \sum_{L \in \mathbb Z^{s+2t}} g_{L}(v) \exp \left( 2\pi\sqrt{-1} \langle A \cdot L \vert a \rangle \right) \\
&& \cdot \left((AL)_{s+j}+\sqrt{-1}(AL)_{s+t+j}\right) 
\end{eqnarray*}
We notice that, for any $L$, there is at least one $j$ such that $(AL)_{s+j}+\sqrt{-1}(AL)_{s+t+j}\neq0$, since the columns of $A$ are linearly independent over $\mathbb Q$. Therefore we can set $g_0=0$ (up to an additive constant) and, for $L\in\mathbb Z^{s+2t}\setminus 0$,
$$
g_L := \frac{1}{\pi\sqrt{-1}} \left((AL)_{s+j}+\sqrt{-1}(AL)_{s+t+j}\right)^{-1} f_{j,L} , 
$$
and there is no ambiguity in the choice of such a $j$ because of \eqref{eq:compatibility}.

It remains to prove that the formal solution $g=\sum g_L \exp(2\pi\sqrt{-1}\langle AL \vert a\rangle)$ is actually smooth: that is, that the Fourier coefficients $g_L$ decay faster than any power $\|L\|^{-N}$ for $N > 0$, as $\|L\| \to \infty$.

We make use of the following application of the Subspace Theorem \cite{schmidt}, generalizing the Roth Theorem for $n=1$:

\begin{thm}[{Schmidt \cite{schmidt}, see {\itshape e.g.}\ \cite[Theorem 7.3.2]{bombieri-gubler}}]
Let $\alpha_0, \ldots, \alpha_n$ be algebraic numbers. Then, for every $\varepsilon>0$, the inequality
$$ 0 < | \alpha_0q_0+\cdots+\alpha_n q_n| < (\max\{|q_0|,\ldots,|q_n|\})^{-n-\varepsilon} $$
has only finitely-many solution $(q_0,\ldots,q_n) \in \mathbb Z^{n+1}$.
\end{thm}

We recall that $A_k^h$, for $k\in\{s+1, \ldots, s+2t\}$ and $h\in\{1,\ldots,s+2t\}$, are algebraic numbers. Moreover, once fixed $L\in\mathbb Z^{s+2t}\setminus 0$, there exists $j\in\{1,\ldots,t\}$ such that $(AL)_{s+j}+\sqrt{-1}(AL)_{s+t+j}\neq0$, namely, there exists $k\in\{1,\ldots,2t\}$ such that $(AL)_{s+k}\neq0$, and we can take any such $j$'s for defining $g_L$ thanks to the compatibility condition \eqref{eq:compatibility}.
Then, for $L\in\mathbb Z^{s+2t}\setminus 0$:
\begin{eqnarray*}
|g_L| &=& \frac{1}{\pi} \left|(AL)_{s+j}+\sqrt{-1}(AL)_{s+t+j}\right|^{-1} |f_{j,L}| \\
&\leq& \frac{1}{\pi} \frac{1}{ \left|(AL)_{s+k}\right| } |f_{j,L}|
= \frac{1}{\pi} \frac{1}{ \left|A_{s+k}^1L_1+\cdots+A_{s+k}^{s+2t}L_{s+2t}\right| } |f_{j,L}| \\
&<& \frac{c}{\pi} \max\{|L_1|, \ldots, |L_{s+2t}|\}^{n+\varepsilon} |f_{j,L}| \\
&\leq& \frac{c}{\pi} \|L\|^{n+\varepsilon} |f_{j,L}|,
\end{eqnarray*}
where $c$ is a positive constant depending just on $A$ and on the fixed $\varepsilon>0$, and independent of $L$ and of the chosen $j$ and $k$.
We have then proven that the Fourier coefficients $g_L$ decay as
\begin{equation}\label{eq:fast}
|g_L| < \frac{c}{\pi} \cdot \|L\|^{n+\varepsilon} \cdot |f_{j,L}|.
\end{equation}
Since the form $\alpha$ is smooth, the fast decay is satisfied by the Fourier coefficients $f_{j,L}$ of the $f_j$'s, and together with formula \eqref{eq:fast} this implies the fast decay also for the Fourier coefficients $g_L$ of $g$. Whence, $g$ is a smooth solution.

This finally proves that any class in $H^1(X^{\text{ab}};\mathcal O_{X^{\text{ab}}})$ has a unique flat representative, and by the above argument we get that $H^1(X^{\text{ab}};\mathcal O_{X^{\text{ab}}})^U=0$.
\end{proof}

\begin{proof}[Proof of Claim H2.]
First of all, we argue as we did for diagram~\eqref{eq:five-term-spectral-O}, the only difference being that this time we forget the holomorphic structure. Namely, we use $\mathcal{F}=\C_X$ instead of $\mathcal{F}=\mathcal{O}_X$. Everything works the same way, thanks to $H^j(\tilde X;\C_{\tilde X})=0$ for any $j\geq 1$. Denoting by $S:=H^0(\tilde X;\C_{\tilde X})$, the Lyndon-Hochschild-Serre spectral sequence reads
$$ E^{p,q}_2=H^p(U;H^q(\mathcal O_K;S))\Rightarrow H^{p+q}(\pi_1(X);S)\;,$$
and the associated cohomology five-term exact sequence yields
 $$
\xymatrix{
  0 \ar[r] & \C^s \ar[r] & H^{1}(X;\C_X) \ar[r] & H^1(X^{\text{ab}};\C_{X^{\text{ab}}})^U \ar[dll] \\
& H^2(U;\C_U) \ar[r] & H^{2}(X;\C_X)\;.
 } $$

The map $\C_{\tilde X} \to \mathcal{O}_{\tilde X}$ induces a map $S\to R$, and hence a morphism of exact sequences
 $$ 
\resizebox{.9\hsize}{!}{$\xymatrix{
  0 \ar[r] \ar@{=}[d] & \C^s \ar[r] \ar@{=}[d] & H^{1}(X;\C_X) \ar[rr]^{0\quad} \ar[d] & & H^1(X^{\text{ab}};\C_{X^{\text{ab}}})^U \ar[r] \ar[d] & H^2(U;\C_U) \ar[r] \ar@{=}[d] & H^{2}(X;\C_X) \ar[d]^{q} \ar[r] & 0 \\
  0 \ar[r] & \C^s \ar[r] & H^{1}(X;\mathcal O_X) \ar[rr]_{0\quad} & & H^1(X^{\text{ab}};\mathcal O_{X^{\text{ab}}})^U \ar[r] \ar[d]
    & H^2(U;\C_U) \ar[r] & H^{2}(X;\mathcal O_X) & \\
  & & & & 0 & & &
 } $}$$
Here, we used that: by Claim (H1), we have that the map $H^1(X;\mathcal O_X)\to H^1(X^{\text{ab}};\mathcal O_{X^{\text{ab}}})^U$ is the zero map; by \cite[Proposition 2.3]{oeljeklaus-toma}, we have $b_1=s$, so the map $H^{1}(X;\C_X) \to H^1(X^{\text{ab}};\C_{X^{\text{ab}}})^U$ is the zero map, too; again by Claim (H1'), the map $H^1(X^{\text{ab}};\C_{X^{\text{ab}}})^U \to H^1(X^{\text{ab}};\mathcal O_{X^{\text{ab}}})^U$ is surjective.
Finally, the map $H^2(U;\C_U) \to H^{2}(X;\C_X)$ is surjective: indeed, we claim that the map $H^2(U;\C_U) \to E^{2,0}_\infty$ is surjective and $E^{0,2}_{2}=0=E^{1,1}_{2}$.
This follows by \cite[pages 166--167]{oeljeklaus-toma} in the case when $X$ is of simple type. In fact, thanks to \cite[Theorem 3.1]{istrati-otiman}, we just need that there are no embeddings $\sigma_j$ and $\sigma_k$, for $j,k\in\{1,\ldots,s+2t\}$ with $j\neq k$, such that $\sigma_j(u)\sigma_k(u)=1$ for any $u \in U$; Oeljeklaus-Toma manifolds of simple type satisfy this condition, see \cite[page 16]{istrati-otiman}. We claim that this latter property always holds true, even when $X$ is not of simple type. Indeed, consider $K':=\mathbb Q[U]$. Since $U$ is still admissible for defining an Oeljeklaus-Toma manifold $X':=X'(K',U)$, see \cite[Lemma 1.6]{oeljeklaus-toma}, and $X'$ is of simple type, therefore there are no embeddings $\sigma_j$ and $\sigma_k$ of $K'$, with $j\neq k$, such that $\sigma_j(u)\sigma_k(u)=1$ for any $u \in U$. Moreover, there is no embedding $\sigma_j$ of $K'$ such that $\sigma_j^2(u)=1$ for any $u \in U$. Since the embeddings of $K'$ are just the restrictions of the embeddings of $K$, we get the claim.

At the end, we get that $q$ is injective by diagram chasing.
\end{proof}

This completes the proof of Theorem~\ref{thm:flat} by proving that any line bundle on an Oeljeklaus-Toma manifold is flat.
\end{proof}

We note that the argument in the last lines of the previous proof shows that $\rho_2=0$ in the notation of \cite[Theorem 3.1]{istrati-otiman}, therefore, thanks to Istrati and Otiman's result, we get the following, generalizing \cite[Proposition 2.3]{oeljeklaus-toma}:

\begin{prop}\label{rmk:b2}
Any Oeljeklaus-Toma manifold has $b_2={s \choose 2}$, where $s$ is the number of real embeddings.
\end{prop}

A well-known result by Ornea and Verbitsky \cite{ornea-verbitsky} for $t=1$ and, for any $s,t$, by Battisti and Oeljeklaus \cite{battisti-oeljeklaus}, states that Oeljeklaus-Toma manifolds have no divisors.
In the following proposition, we show that this result follows from Theorem \ref{thm:flat}.

\begin{prop}[{Battisti and Oeljeklaus \cite[Theorem 3.5]{battisti-oeljeklaus}}]
Let $X=X(K, U)$ be an Oeljeklaus-Toma manifold. Then $X$ has no divisors.
\end{prop}

\begin{proof}
Take any line bundle on $X$, which is then flat, and let $\rho$ be the associated representation. But any representation $\rho\colon \pi_1(X)\to \C^*$ induces the identity on $\mathcal{O}_K$ \cite[Proposition 6]{braunling}. Therefore the pull-back of $L_\rho$ to $X^{\text{ab}}$ is trivial, and its sections are constants. Therefore $L_\rho$ has no non-trivial sections on $X$.
\end{proof}

\section{Rigidity of Oeljeklaus-Toma manifolds}

In this section we extensively apply techniques similar to the ones used in Section~\ref{sec:flat}, to prove the following vanishing result. 

\begin{thm}\label{thm:vanishing}
Let $X=X(K,U)$ be an Oeljeklaus-Toma manifold. Take any faithful representation $\rho\colon U\to\C^*$, and let $L_\rho$ be its associated flat holomorphic line bundle on $X$. Then $H^1(X;L_\rho)=0$ unless $\rho=\bar{\sigma}_ i^{-1}$ for some $i\{t+1,\dots t+s\}$.
\end{thm}

\begin{proof}
We use group cohomology, with the action of $U\ni u$ on $R=H^0(\tilde X;\mathcal{O}_{\tilde X})$ given by
$$ L_u(f):=\rho(u)f\circ  R_u ,$$
where $R_u$ is the rotation given by equation \eqref{eq:rot}, see Section \ref{sec:flat}. Since the pull-back of $L_\rho$ to $\tilde{X}$ is trivial, we get
\begin{equation*}
H^1(\mathcal O_K\rtimes U;R)\simeq H^1(X;L_\rho).
\end{equation*}
From the Lyndon-Hochschild-Serre spectral sequence and the cohomology five-term exact sequence we obtain, as in diagram \eqref{eq:five-term-spectral-O}, the exact sequence
$$ \xymatrix{ H^1(U;H^0(\mathcal{O}_K;R)) \ar[r] & H^1(X;L_\rho) \ar[r] & H^1(\mathcal{O}_K;R)^U }.$$

On the one side,
$ H^1(U;H^0(\mathcal{O}_K;R)) =0$
since $\rho$ is faithful and $U$ is free Abelian (this follows easily from the fact that, for a free cyclic group $U$, one has $H^0(U;\C)=H^1(U; \C)=0$ for any non-trivial representation $\rho$, and then performing induction on the rank of $U$ using  again the Lyndon-Hochschild-Serre spectral sequence).
On the other side, we have $H^1(\mathcal{O}_K;R)^U=H^1(X^{\text{ab}};\mathcal{O}_{X^{\text{ab}}})^U.$ 
But for any $u\in U$ we have 
$$L_u^*(d\overline{z}_i)=\rho(u)\overline{\sigma}_i(u)d\overline{z}_i , $$
hence the conclusion. 
\end{proof}

\begin{rmk}
Another possible argument for Theorem \ref{thm:vanishing} may be found on elliptic Hodge theory, as suggested in \cite{tomassini-torelli}.
We just notice that, if $\vartheta$ is the closed $1$-form determined by $\rho$ as $\rho(\gamma)=\exp\int_\gamma\vartheta$, then the (de Rham) cohomology of $X$ with values in the complex line bundle $L_\rho$ corresponds to the cohomology of the trivial bundle $X\times\C$ with respect to the flat connection $d_{\vartheta}:=d+\vartheta\wedge\sspace$. We split $d_\vartheta=\overline\partial_\vartheta+\partial_\vartheta$ where $\overline\partial_\vartheta:=\overline\partial+\vartheta^{0,1}\wedge\sspace$. Here, $\vartheta^{0,1}$ is the $(0,1)$-component of $\vartheta$. The (Dolbeault) cohomology of $X$ with values in the holomorphic line bundle $L_\rho$ corresponds to the cohomology of the trivial bundle with respect to the flat connection $\overline\partial_\vartheta$.
Elliptic Hodge theory applies with the operator $[\overline\partial_\vartheta,\overline\partial_\vartheta^*]$. Note indeed that the operator is elliptic, since the second-order part of it is equal to the second-order part of $[\overline\partial,\overline\partial^*]$. We claim that the zeroth-order part of $[\overline\partial_\vartheta,\overline\partial_\vartheta^*]$ is positive (with respect to the $L^2$-pairing).
Indeed, note that $\overline\partial_\vartheta^*=-*\overline\partial_{-\vartheta}*$. Therefore the zeroth-order term is given by $\vartheta^{0,1}\wedge *(\vartheta^{0,1}\wedge *\sspace)+*(\vartheta^{0,1}\wedge *(\vartheta^{0,1}\wedge\sspace))$. Note that, on $1$-forms $\gamma$, it holds $\left\langle \vartheta^{0,1}\wedge *(\vartheta^{0,1}\wedge *\gamma) \middle\vert \gamma \right\rangle = \left\|\vartheta^{0,1}\wedge *\gamma \right\|^2\geq0$, and, similarly, $\left\langle *(\vartheta^{0,1}\wedge *(\vartheta^{0,1}\wedge\gamma)) \middle\vert \gamma \right\rangle=\left\|\vartheta^{0,1}\wedge\gamma \right\|^2\geq 0$.
\end{rmk}

As a corollary, we get rigidity in the sense of the theory of deformations of complex structures of Kodaira-Spencer-Nirenberg-Kuranishi. See \cite[Proposition 2]{inoue} for rigidity in the case $s=t=1$ of Inoue-Bombieri surfaces.

\begin{cor}\label{cor:rigid}
 Oeljeklaus-Toma manifolds of simple type are rigid.
\end{cor}

\begin{proof}
Note that $\Theta_{\mathbb{H}^s\times\C^t}=\left\langle\frac{\partial}{\partial w^1},\ldots,\frac{\partial}{\partial w^s}, \frac{\partial}{\partial z^1},\ldots,\frac{\partial}{\partial z^t}\right\rangle$, and $\mathcal{O}_K \rtimes U\ni (a,u)$ acts on $\frac{\partial}{\partial w^h}$, respectively $\frac{\partial}{\partial z^k}$, as multiplication by $\sigma_h(u)$, respectively $\sigma_{s+k}(u)$. Whence the holomorphic tangent bundle of an Oeljeklaus-Toma manifold splits as
$$ \Theta_X = \bigoplus_{j=1}^{s+t} L_{\sigma_j^{{-1}}} , $$
where $L_{\sigma_j}$ are the line bundle associated to the embeddings $\sigma_j$.
By Theorem \ref{thm:vanishing}, we get $H^1(X;\Theta_X)=0$,
unless $\sigma_i^{{-1}}(u)=\overline{\sigma}_j^{-1}(u)$ for some $i \in \{1,\dots s+t\}$, $j \in \{s+1,\dots, s+t\}$ and any $u\in U.$ But this would imply that for all $u\in U$ we have $\sigma_i(u)=\sigma_{j+t}(u)$ hence all $u\in U$ live in a proper subfield of $K$, absurd since we assumed $X$ to be of simple type.
This  proves the claim.
\end{proof}

\begin{rmk}
For the case $t=1$, a stronger result was obtained by O. Braunling. He proves in \cite[Proposition 1]{braunling} that, if two Oeljeklaus-Toma manifolds $X'=X(K';\mathcal{O}_{K'}^{*,+})$ and $X''=X(K'';\mathcal{O}_{K''}^{*,+})$, both having $t=1$, are homotopy equivalent, then they are isomorphic.
\end{rmk}

\bibliographystyle{alpha}
\bibliography{bibliografia} 

\end{document}